\documentclass[a4paper, 12pt, reqno]{amsart} 
\allowdisplaybreaks
\usepackage{amsmath,amssymb,amsthm}
\usepackage{microtype}
\usepackage[left=2.5cm, right=2.5cm, top=3cm, bottom=3cm]{geometry} 
\usepackage[colorlinks = true, urlcolor = blue, linkcolor = blue,
            citecolor = blue]{hyperref}
\usepackage{xcolor}
\usepackage[shortlabels]{enumitem}

\newtheorem{theorem}{Theorem}
\newtheorem{lemma}{Lemma}
\newtheorem{corollary}{Corollary}
\newtheorem{proposition}{Proposition}
\theoremstyle{definition}
\newtheorem{definition}{Definition}

\newtheorem{assumption}{Assumption}

\newcommand{\E}{\mathbb{E}}

\newcommand{\R}{\mathbb{R}}

\newcommand{\F}{\mathcal{F}}

\newcommand{\Pot}{\mathcal{P}}

\begin{document}
\title[Clark-Kushner condition for interacting reinforced random walks]{The Clark-Kushner condition for interacting reinforced random walks on finite graphs}

\author[F. P. A. Prado]{Fernando P. A. Prado$^{*}$}

\address[F. P. A. Prado]{Departamento de Computa\c{c}\~ao e
  Matem\'atica, Universidade de S\~ao Paulo, 
  Avenida Bandeirantes 3900, Ribeir\~ao Preto, S\~ao Paulo, 
  14040-901, Brasil}
\email{feprado@usp.br}

\author[R. A. Rosales]{Rafael A. Rosales}

\address[R. A. Rosales]{Departamento de Computa\c{c}\~ao e
  Matem\'atica, Universidade de S\~ao Paulo, 
  Avenida Bandeirantes 3900, Ribeir\~ao Preto, S\~ao Paulo, 
  14040-901, Brasil}
\email{rrosales@usp.br}

\date{\today}

\begin{abstract}
We establish the Clark-Kushner condition for a large class of interacting vertex-reinforced random walks on finite
graphs, where the transition matrix $Q^i(x)$ of each walk depends on the
joint vector $x$ of vertex occupation proportions and may have distinct
rows. This allows one to study the dynamics of the vertex occupation
measure by using the tools of stochastic approximation theory. However,
the standard approach fails because the noise inputs are in our case not
a martingale difference: they retain memory of the previous state. Using
the solution of the Poisson equation for Markov chains, we decompose the
noise into a martingale difference minus the increment of a bounded
process---a structure originating in Gordin's work on limit theorems for
stationary processes. The key technical ingredient of our approach is a
uniform geometric ergodicity bound derived from the Dobrushin contraction
coefficient, which also controls the Lipschitz continuity of the solution
of the Poisson equation. Our hypotheses require only that each $Q^i(x)$ be
irreducible, aperiodic, and Lipschitz continuous in $x$; in particular,
strictly positive entries are not assumed. Our results generalize and
simplify previous arguments considered for single self-reinforced
vertex-reinforced random walks.
\end{abstract}

\maketitle

{\renewcommand{\thefootnote}{}\footnotetext{$^{*}$Corresponding author: \texttt{feprado@usp.br}}}

%==============================================================================
\section{Introduction}\label{sec:intro}
%==============================================================================

Reinforced random walks form a class of stochastic processes in which
transition probabilities depend on the accumulated history of visits.
The defining mechanism is self-reinforcement: visiting a state modifies
the likelihood of future transitions to that state. This feedback renders
the process non-Markovian, and understanding its long-term behaviour
requires techniques beyond classical Markov chain theory.

The study of reinforced walks was initiated by Coppersmith and
Diaconis~\cite{CD87} for edge-reinforced processes, and by
Pemantle~\cite{P92} for vertex-reinforced walks. A central phenomenon
is localisation: under appropriate conditions, the walk eventually
confines itself to a strict subset of vertices. This has been established
for various graph structures and reinforcement schemes;
see~\cite{B97, PV99, V01, T04, BT11, BSS14, CT2017} and references
therein.

More recently, attention has turned to systems of \emph{multiple
interacting} reinforced walks, where the transition probabilities of
each walk depend on the joint history of all walks. Such models arise
naturally in applications where agents learn from shared experience.
The works~\cite{RPP22, PCR2023, PR2025} develop frameworks for
analysing these interacting systems via stochastic approximation.

The general framework of stochastic approximation with Markov-chain
noise controlled by the iterand was established by M\'etivier and
Priouret~\cite{MP84, MP87}. They prove almost sure convergence for
algorithms of the form
$\theta_{n+1} = \theta_n + \gamma_{n+1} f(\theta_n, Y_{n+1})$,
where $(Y_n)$ is a Markov chain whose transition kernel depends on the
current value $\theta_n$ of the iterand, under a global Lyapunov
condition for the associated mean-field equation and assuming positive
recurrence of the controlled chain for each fixed parameter value.
The present work instantiates this framework with $\theta_n = X(n)$
(the joint occupation proportion) and $Y_n = (W^1(n), \ldots, W^m(n))$
(the positions of all walks), and provides an explicit verification of
the required conditions under hypotheses that do not demand strictly
positive entries in $Q^i(x)$.

A key step in this verification is establishing the \emph{Clark-Kushner
condition}~\cite{MP84}, which controls the asymptotic behaviour of
weighted sums of the stochastic input. When the transition matrix has
all rows equal---as occurs when each walk jumps to a vertex chosen
independently of its current position---the stochastic input is a
martingale difference, and the Clark-Kushner condition follows
immediately. However, for general transition matrices with distinct
rows, the input is no longer a martingale difference: the walk
remembers where it was. This memory introduces correlations that must
be carefully controlled.

Our approach employs the \emph{Poisson equation} for Markov chains,
building on ideas originating in Gordin's work~\cite{G69} on limit
theorems for stationary processes. The key contribution of this article is the
derivation of uniform geometric ergodicity---and hence uniform control on
the solution of the Poisson equation---under the sole hypotheses that each
$Q^i(x)$ is irreducible, aperiodic, and Lipschitz continuous in $x$.
Strict positivity of entries is not required. The uniformity follows
from a compactness argument: by Lipschitz continuity, there exists a
uniform integer $k_* \geq 1$ such that $Q^i(x)^{k_*}$ has all entries
bounded below by some $\delta > 0$ for every $x \in \triangle^m$;
the Dobrushin contraction coefficient then yields a geometric bound
$\|Q^i(x)^k - \Pi^i(x)\| \leq C^* \lambda^k$ uniformly in $x$
(see Section~\ref{sec:dobrushin}).

The remainder of the article is organised as follows.
Section~\ref{sec:model} introduces the model, the standing hypotheses, and
the stochastic-approximation reformulation of the problem.
Section~\ref{sec:dobrushin} establishes uniform geometric ergodicity through
the Dobrushin contraction coefficient.
Section~\ref{sec:poisson} develops the solution of the Poisson equation and
the central Lipschitz-type estimate for $Q(x)\Pot(x)$.
Section~\ref{sec:summability} establishes the decomposition of the stochastic
input and the summability of the correction term, and
Section~\ref{sec:main} combines these ingredients into the proof of the main
theorem. Section~\ref{sec:remarks} discusses extensions to several interacting
walks and to independent geometric transition times.

%==============================================================================
\section{Model and hypotheses}\label{sec:model}
%==============================================================================

We consider $m \geq 1$ random walks on a finite connected graph
$G=(V,E)$ with $|V| = d \geq 2$. For clarity, we present the case $m = 1$; the extension to $m \geq 1$ is straightforward (Section~\ref{sec:remarks}).

\subsection{Basic setup}

Let $W(n) \in \{1, \ldots, d\}$ denote the position of the walk at time
$n \geq 1$. Set $\xi(0) = (1, 1, \ldots, 1)$, and for $n > 0$ let
$\xi(n) = e_{W(n)} \in \R^d$ be the indicator vector, where $e_v$ is the
$v$-th standard basis vector. All vectors are row vectors; matrices act
by right multiplication.

The occupation proportion is
\[
X(n) = \frac{1}{d+n}\sum_{\ell=0}^{n}\xi(\ell) \in \triangle,
\]
where $\triangle = \{x \in \R^d : x_v \geq 0, \sum_v x_v = 1\}$ is the
$d-1$ unitary simplex.

The transition probability from vertex $w$ to vertex $v$, given occupation
proportion $x = X(n)$, is $Q_{wv}(x)$. We shall make the following
assumptions about $Q(x) = (Q_{wv}(x))$, $wv \in V$.

\begin{assumption}[Irreducibility and aperiodicity]\label{ass:H1}
For each $x \in \triangle$, the matrix $Q(x)$ is irreducible and aperiodic.
\end{assumption}

\begin{assumption}[Lipschitz continuity]\label{ass:H2}
The map $x \mapsto Q(x)$ is Lipschitz continuous: there exists $L_Q > 0$ such that $\|Q(x) - Q(y)\| \leq L_Q \|x - y\|$ for all $x, y \in \triangle$.
\end{assumption}

Under Assumption~\ref{ass:H1}, each $Q(x)$ admits a unique invariant measure $\pi(x)$. Let $\Pi(x)$ denote the matrix with all rows equal to $\pi(x)$.

\subsection{Evolution equation}

The occupation proportion evolves according to
\begin{equation}\label{eqn:X_evolution}
X(n+1) - X(n) = \gamma_n\big(F(X(n)) + U(n+1)\big),
\end{equation}
where $\gamma_n = 1/(d+n+1)$, $F(x) = -x + \pi(x)$, and $U(n+1) = \xi(n+1) - \pi(X(n))$.

Let $\F_n = \sigma\big(W(0), W(1), \ldots, W(n)\big)$ denote the natural
filtration. Since $X(n)$ is $\F_n$-measurable and, given $\F_n$, the walk
moves from $W(n)$ to $W(n+1)$ according to the row $W(n)$ of $Q(X(n))$, the one-step
conditional law reads
\begin{equation}\label{eqn:one_step}
\E\big[\xi(n+1)\mid\F_n\big] = \xi(n)Q(X(n)).
\end{equation}
In particular, we are in the situation where
$\E[\xi(n+1)\mid\F_n] = \xi(n)Q(X(n)) \neq \pi(X(n))$, and so
$U(n+1)$ is \emph{not} a martingale difference: this is the central
difficulty addressed below.

To verify the Clark-Kushner condition (see~\cite{MP84}), it is sufficient to show that
\begin{equation}\label{eqn:clark_kushner}
\lim_{n \to \infty} \sup_{r \geq n} \left\| \sum_{k=n}^{r-1} \gamma_k U(k+1) \right\| = 0 \quad \text{a.s.}
\end{equation}

\subsection{Norms}

We use the $\ell^1$ norm: $\|v\| = \sum_{j=1}^{d}|v_j|$ for vectors, and the induced operator norm $\|A\| = \max_i \sum_j |A_{ij}|$ for matrices. This norm is submultiplicative and satisfies $\|Q(x)\| = 1$ for stochastic matrices.

In order to verify~\eqref{eqn:clark_kushner}, we use that $\gamma_k = O(1/k)$
and establish the decomposition
\[
U(n+1) = D(n+1) - \big(H(n+1) - H(n)\big) + C(n+1),
\]
where $D(n+1)$ is a bounded martingale difference, $H(n)$ is a bounded
sequence, and the correction term satisfies
$\|C(k+1)\| = O\!\left((\log k)^2/k\right)$; the precise statement is
equation~\eqref{eqn:decomp_full} in Section~\ref{sec:summability}. Granting
this decomposition, the sum $\sum_k \gamma_k U(k+1)$ splits into a convergent
martingale series, a telescoping term handled by summation by parts, and a
summable remainder. The intervening sections develop the definitions and
estimates required to make this precise: uniform geometric ergodicity
(Section~\ref{sec:dobrushin}), the solution of the Poisson equation and the
variation bound for $Q(x)\Pot(x)$ (Section~\ref{sec:poisson}).

%==============================================================================
\section{Uniform contraction via the Dobrushin coefficient}\label{sec:dobrushin}
%==============================================================================

\subsection{The Dobrushin coefficient}

For a stochastic matrix $Q$, the \emph{Dobrushin contraction coefficient} is
\[
\alpha(Q) = \frac{1}{2}\max_{i,j}\sum_v |Q_{iv} - Q_{jv}| = 1 - \min_{i,j}\sum_v \min(Q_{iv}, Q_{jv}).
\]
It measures the maximal total variation distance between rows of $Q$.
Key properties (see~\cite{Seneta81}, Sections~3.1--3.4):
\begin{enumerate}[(i)]
\item $0 \leq \alpha(Q) \leq 1$;
\item $\alpha(Q_1 Q_2) \leq \alpha(Q_1)\alpha(Q_2)$ (submultiplicativity);
\item For any probability vectors $\mu, \nu$:
      $\|\mu Q - \nu Q\| \leq \alpha(Q)\|\mu - \nu\|$.
\end{enumerate}

\subsection{Uniform geometric ergodicity}

\begin{proposition}\label{prop:uniform_contraction}
Under Assumptions~\ref{ass:H1}--\ref{ass:H2}, there exist constants $C^* > 0$, $\lambda \in (0,1)$, and an integer $k_* \geq 1$ such that for all $x \in \triangle$ and $k \geq 0$:
\begin{equation}\label{eqn:uniform_decay}
\|Q(x)^k - \Pi(x)\| \leq C^* \lambda^k.
\end{equation}
\end{proposition}

\begin{proof}
By Assumption~\ref{ass:H1} and compactness of $\triangle$, there exists $k_*$ such that $Q(x)^{k_*}$ has all entries bounded below by some $\delta > 0$ uniformly in $x$. Then $\alpha(Q(x)^{k_*}) \leq 1 - d\delta =: \mu < 1$ uniformly. For $k = qk_* + r$ with $0 \leq r < k_*$:
\[
\|Q(x)^k - \Pi(x)\| \leq \|Q(x)^r\| \cdot \|Q(x)^{qk_*} - \Pi(x)\| \leq 2\mu^q \leq 2\mu^{-1}\mu^{k/k_*} = C^*\lambda^k,
\]
with $\lambda = \mu^{1/k_*}$ and $C^* = 2/\mu$.
\end{proof}

%==============================================================================
\section{The Poisson equation}\label{sec:poisson}
%==============================================================================

\subsection{The potential matrix}

For each $x \in \triangle$, the \emph{Poisson equation} associated with
$Q(x)$ is
\begin{equation}\label{eqn:poisson}
(I - Q(x))\,\Pot = I - \Pi(x).
\end{equation}

\begin{definition}
The \emph{potential matrix} $\Pot(x)$ is the solution of the Poisson
equation~\eqref{eqn:poisson}, given explicitly by
\[
\Pot(x) = \sum_{j=0}^{\infty}\big(Q(x)^j - \Pi(x)\big).
\]
\end{definition}

By Proposition~\ref{prop:uniform_contraction}, this series converges uniformly in $x$, with
\begin{equation}\label{eqn:P_bound}
\|\Pot(x)\| \leq C_\Pot := \frac{C^*}{1 - \lambda}.
\end{equation}

\begin{lemma}\label{lem:poisson_properties}
For each $x \in \triangle$:
\begin{enumerate}[(a)]
\item $Q(x)\Pot(x) = \Pot(x)Q(x)$;
\item $(I - Q(x))\Pot(x) = \Pot(x)(I - Q(x)) = I - \Pi(x)$;
\item $\Pot(x)\mathbf{1}^T = 0$.
\end{enumerate}
\end{lemma}

\begin{proof}
For part (a), note that $Q(x)\Pi(x) = \Pi(x)Q(x) = \Pi(x)$. Hence
\[
Q(x)\sum_{j=0}^{n}(Q(x)^j - \Pi(x)) = \sum_{j=0}^{n}(Q(x)^{j+1} - \Pi(x)) = \sum_{j=0}^{n}(Q(x)^j - \Pi(x))Q(x).
\]
Passing to the limit as $n \to \infty$ gives $Q(x)\Pot(x) = \Pot(x)Q(x)$.

For part (b), the first equality follows immediately from (a). For the second equality:
\[
(I - Q(x))\Pot(x) = \sum_{j=0}^{\infty}(Q(x)^j - Q(x)^{j+1}) = \lim_{n\to\infty}(I - Q(x)^{n+1}) = I - \Pi(x),
\]
where the second equality follows by telescoping and the third uses $Q(x)^n \to \Pi(x)$.

Part (c) follows from $Q(x)^j\mathbf{1}^T = \Pi(x)\mathbf{1}^T = \mathbf{1}^T$.
\end{proof}

\subsection{Lipschitz continuity of the invariant measure}

\begin{lemma}\label{lem:pi_lipschitz}
Under Assumptions~\ref{ass:H1}--\ref{ass:H2}, $x \mapsto \pi(x)$ is Lipschitz with constant $L_\pi$.
\end{lemma}

\begin{proof}
By Cramer's rule, $\pi_v(x)$ is a ratio of determinants involving entries of $Q(x)$. The numerator and denominator are polynomials in these entries, hence Lipschitz in $x$. The denominator is bounded away from zero by Assumption~\ref{ass:H1} and compactness of $\triangle$.
\end{proof}

\subsection{Estimates for the variation of  $Q(x)\Pot(x)$}

Since $Q(x)\Pi(x) = \Pi(x)$, we have
\begin{equation}\label{eqn:QP_series}
Q(x)\Pot(x) = \sum_{j=0}^{\infty}\big(Q(x)^{j+1} - \Pi(x)\big).
\end{equation}
Using this identity, we now bound the variation of $Q(x)\Pot(x)$ in $x$.

\begin{lemma}\label{lem:QP_lipschitz}
There exists a constant $L_{Q\Pot}$ such that for all $x, y \in \triangle$:
\[
\|Q(x)\Pot(x) - Q(y)\Pot(y)\| \leq L_{Q\Pot}\big(1 + |\log\|x-y\||\big)^2\|x - y\|.
\]
\end{lemma}

\begin{proof}
We estimate term by term:
\[
Q(x)\Pot(x) - Q(y)\Pot(y) = \sum_{j=0}^{\infty}\big[(Q(x)^{j+1} - \Pi(x)) - (Q(y)^{j+1} - \Pi(y))\big].
\]

For each $j$, we have two bounds.

\emph{Geometric bound:} By the triangle inequality and Proposition~\ref{prop:uniform_contraction}:
\begin{equation}\label{eqn:geometric_bound}
\|(Q(x)^{j+1} - \Pi(x)) - (Q(y)^{j+1} - \Pi(y))\| \leq 2C^*\lambda^{j+1}.
\end{equation}

\emph{Lipschitz bound:} By telescoping $Q(x)^{j+1} - Q(y)^{j+1} = \sum_{\ell=0}^{j} Q(x)^\ell(Q(x) - Q(y))Q(y)^{j-\ell}$ and $\|Q(x)^\ell\| \leq 1$:
\[
\|Q(x)^{j+1} - Q(y)^{j+1}\| \leq (j+1)L_Q\|x - y\|.
\]
Combined with $\|\Pi(x) - \Pi(y)\| \leq L_\pi\|x - y\|$ (Lemma~\ref{lem:pi_lipschitz}):
\begin{equation}\label{eqn:lipschitz_bound}
\|(Q(x)^{j+1} - \Pi(x)) - (Q(y)^{j+1} - \Pi(y))\| \leq \big((j+1)L_Q + L_\pi\big)\|x - y\|.
\end{equation}

We split the series at a cutoff $j_* \geq 0$ to be chosen. For $j \leq j_*$, we use~\eqref{eqn:lipschitz_bound}; for $j > j_*$, we use~\eqref{eqn:geometric_bound}.

\emph{Sum for $j \leq j_*$:} Using~\eqref{eqn:lipschitz_bound},
\begin{align}
\sum_{j=0}^{j_*} \big((j+1)L_Q + L_\pi\big)\|x-y\| 
&= \left( L_Q \sum_{j=0}^{j_*}(j+1) + (j_*+1)L_\pi \right)\|x-y\| \notag\\
&= \left( L_Q \frac{(j_*+1)(j_*+2)}{2} + (j_*+1)L_\pi \right)\|x-y\| \notag\\
&\leq K_1 (j_*+1)^2 \|x-y\|, \label{eqn:sum_lipschitz}
\end{align}
where $K_1 = L_Q + L_\pi$.

\emph{Sum for $j > j_*$:} Using~\eqref{eqn:geometric_bound},
\begin{equation}\label{eqn:sum_geometric}
\sum_{j > j_*} 2C^*\lambda^{j+1} = \frac{2C^*\lambda^{j_*+2}}{1-\lambda} = K_2 \lambda^{j_*+1},
\end{equation}
where $K_2 = 2C^*\lambda/(1-\lambda)$.

\emph{Choice of cutoff:} Set $j_* = \lfloor |\log\|x-y\||/|\log\lambda| \rfloor$, assuming $\|x-y\| < 1$. Since $j_* + 1 \geq |\log\|x-y\||/|\log\lambda|$, we have $(j_*+1)\log\lambda \leq \log\|x-y\|$ (note: $\log\lambda < 0$), hence $\lambda^{j_*+1} \leq \|x-y\|$. By~\eqref{eqn:sum_geometric}:
\[
\sum_{j > j_*} 2C^*\lambda^{j+1} \leq K_2 \|x-y\|.
\]

Also, $j_* + 1 \leq 1 + |\log\|x-y\||/|\log\lambda|$, so by~\eqref{eqn:sum_lipschitz}:
\[
\sum_{j=0}^{j_*} \big((j+1)L_Q + L_\pi\big)\|x-y\| \leq K_1 \left(1 + \frac{|\log\|x-y\||}{|\log\lambda|}\right)^2 \|x-y\|.
\]

Combining both sums, we obtain
\[
\|Q(x)\Pot(x) - Q(y)\Pot(y)\| \leq L_{Q\Pot}\big(1 + |\log\|x-y\||\big)^2\|x - y\|,
\]
where $L_{Q\Pot}$ depends only on $C^*$, $\lambda$, $L_Q$, and $L_\pi$.
\end{proof}

%==============================================================================
\section{Decomposition of the stochastic input}\label{sec:summability}
%==============================================================================

We are now in a position to establish the decomposition of $U(k+1)$ on which
the verification of the Clark-Kushner condition rests.

\begin{proposition}\label{prop:decomp}
For every $k \geq 0$,
\begin{equation}\label{eqn:decomp_full}
U(k+1) = D(k+1) - \big(H(k+1) - H(k)\big) + C(k+1),
\end{equation}
where
\[
D(k+1) = \xi(k+1)\Pot(X(k)) - \xi(k)Q(X(k))\Pot(X(k)),
\qquad
H(k) = \xi(k)Q(X(k))\Pot(X(k)),
\]
\[
C(k+1) = \xi(k+1)\big[Q(X(k+1))\Pot(X(k+1)) - Q(X(k))\Pot(X(k))\big].
\]
Moreover $D(k+1)$ is a martingale difference with respect to $\F_k$, and
$D$ and $H$ are uniformly bounded: $\|H(k)\| \leq C_\Pot$ and
$\|D(k+1)\| \leq 2C_\Pot$, with $C_\Pot$ as in~\eqref{eqn:P_bound}.
\end{proposition}

\begin{proof}
Recall that $U(k+1) = \xi(k+1) - \pi(X(k))$. Since $\xi(k+1)$ is a probability
vector and $\Pi(X(k))$ has every row equal to $\pi(X(k))$, we have
$\xi(k+1)\Pi(X(k)) = \pi(X(k))$. Hence, by part~(b) of
Lemma~\ref{lem:poisson_properties},
\begin{equation}\label{eqn:U_poisson2}
U(k+1) = \xi(k+1)\big(I - \Pi(X(k))\big) = \xi(k+1)\big(I - Q(X(k))\big)\Pot(X(k)).
\end{equation}
Expanding the right-hand side of~\eqref{eqn:U_poisson2} and adding and
subtracting $\xi(k)Q(X(k))\Pot(X(k))$,
\begin{equation}\label{eqn:DB_split}
U(k+1) = D(k+1) - B(k+1),
\end{equation}
where
\begin{align*}
D(k+1) &= \xi(k+1)\Pot(X(k)) - \xi(k)Q(X(k))\Pot(X(k)),\\
B(k+1) &= \xi(k+1)Q(X(k))\Pot(X(k)) - \xi(k)Q(X(k))\Pot(X(k)).
\end{align*}
The term $D(k+1)$ is a martingale difference: $\Pot(X(k))$ is
$\F_k$-measurable, so by the one-step law~\eqref{eqn:one_step},
\[
\E\big[D(k+1)\mid\F_k\big]
= \xi(k)Q(X(k))\Pot(X(k)) - \xi(k)Q(X(k))\Pot(X(k)) = 0.
\]
For the term $B(k+1)$, the term being subtracted is exactly $H(k)$,
while the first term $\xi(k+1)Q(X(k))\Pot(X(k))$ differs from
$H(k+1) = \xi(k+1)Q(X(k+1))\Pot(X(k+1))$ only through the argument
$X(k+1)$ versus $X(k)$. Adding and subtracting $H(k+1)$,
\begin{align}
B(k+1)
&= \xi(k+1)Q(X(k))\Pot(X(k)) - \xi(k)Q(X(k))\Pot(X(k)) \notag\\
&= \big[H(k+1) - H(k)\big]
   - \xi(k+1)\big[Q(X(k+1))\Pot(X(k+1)) - Q(X(k))\Pot(X(k))\big] \notag\\
&= \big[H(k+1) - H(k)\big] - C(k+1).
\label{eqn:B_split}
\end{align}
Substituting~\eqref{eqn:B_split} into~\eqref{eqn:DB_split} yields the
decomposition~\eqref{eqn:decomp_full}.

Finally, using $\|\xi(k)\| = 1$, $\|Q(x)\| = 1$ and the
bound~\eqref{eqn:P_bound},
\[
\|H(k)\| \leq \|\xi(k)\|\,\|Q(X(k))\|\,\|\Pot(X(k))\| \leq C_\Pot,
\]
and likewise $\|D(k+1)\| \leq \|\xi(k+1)\Pot(X(k))\| + \|H(k)\| \leq 2C_\Pot$.
\end{proof}

The decomposition~\eqref{eqn:decomp_full} reduces the verification of the
Clark-Kushner condition to three convergence results: (i) martingale
convergence for $\sum_k \gamma_k D(k+1)$, (ii) summation by parts for
$\sum_k \gamma_k (H(k+1) - H(k))$, and (iii) summability of
$\sum_k \gamma_k C(k+1)$. The first two follow by standard arguments
(Proposition~\ref{prop:decomp} provides the boundedness they require); the
key technical step is~(iii).

\begin{lemma}\label{lem:C_summable}
$\sum_{k=0}^{\infty} \gamma_k \|C(k+1)\| < \infty$ almost surely.
\end{lemma}

\begin{proof}
Since $\gamma_k = 1/(d+k+1)$, we have $\|X(k+1) - X(k)\| \leq 2\gamma_k = 2/(d+k+1)$. By Lemma~\ref{lem:QP_lipschitz}:
\[
\|C(k+1)\| \leq L_{Q\Pot}(1 + |\log(2/(d+k+1))|)^2 \cdot \frac{2}{d+k+1} = O\left(\frac{(\log k)^2}{k}\right).
\]
Multiplying by $\gamma_k = 1/(d+k+1)$:
\[
\gamma_k \|C(k+1)\| = O\left(\frac{(\log k)^2}{k^2}\right).
\]
As $\sum_k (\log k)^2/k^2 < \infty$, the assertion made by the lemma follows.
\end{proof}

%==============================================================================
\section{Proof of the main theorem}\label{sec:main}
%==============================================================================

\begin{theorem}\label{thm:main}
Under Assumptions~\ref{ass:H1}--\ref{ass:H2}, the Clark-Kushner condition~\eqref{eqn:clark_kushner} holds almost surely.
\end{theorem}

\begin{proof}
The decomposition~\eqref{eqn:decomp_full}, established in
Proposition~\ref{prop:decomp}, expresses $U(k+1)$ as a martingale difference
$D(k+1)$ minus the increment of a bounded process $H(k)$ plus a correction
term $C(k+1)$. Multiplying by $\gamma_k$ and summing:
\[
\sum_{k=n}^{r-1}\gamma_k U(k+1) = \sum_{k=n}^{r-1}\gamma_k D(k+1) - \sum_{k=n}^{r-1}\gamma_k(H(k+1) - H(k)) + \sum_{k=n}^{r-1}\gamma_k C(k+1).
\]
Each term vanishes uniformly in $r \geq n$ as $n \to \infty$: the
first by the martingale convergence theorem and Kronecker's lemma
(since $D(k+1)$ is bounded and $\sum_k \gamma_k^2 < \infty$), the
second by summation by parts (since $H(k)$ is bounded and
$\gamma_k \to 0$), and the third by Lemma~\ref{lem:C_summable}.
\end{proof}

\begin{corollary}\label{cor:convergence}
Suppose Assumptions~\ref{ass:H1}--\ref{ass:H2} hold. If, in addition, there exists a Lyapunov function $L: \triangle \to [0, \infty)$ for the mean field $F(x) = -x + \pi(x)$ such that  the $L(F^{-1}(0))$ is finite, then the occupation measure $X(n)$ converges almost surely to the set of zeros of $F$.
\end{corollary}

\begin{proof}
This follows from Theorem~\ref{thm:main} and the general theory of stochastic approximation~\cite{B96, B99}.
\end{proof}

%==============================================================================
\section{Concluding remarks}\label{sec:remarks}
%==============================================================================

We have established the Clark-Kushner condition for vertex-reinforced random
walks whose transition mechanism is governed by a matrix $Q(x)$ that depends on
the occupation proportions and need not have equal rows. The only structural
requirements are that $Q(x)$ be irreducible and aperiodic for each $x$ and
Lipschitz continuous in $x$; in particular, no entry of $Q(x)$ is required to
be strictly positive. The argument rests on writing the stochastic input as a
martingale difference minus the increment of a bounded process plus a slowly
varying correction, and on controlling that correction through the solution of
the Poisson equation, whose uniform regularity is in turn a consequence of the
Dobrushin contraction applied to a suitable power $Q(x)^{k_*}$.

The analysis is not confined to a single walk. For $m \geq 1$ interacting
walks with joint occupation vector $X(n) = (X^1(n), \ldots, X^m(n)) \in
\triangle^m$, each walk carries its own transition matrix $Q^i(x)$ depending on
the joint state, and Assumptions~\ref{ass:H1}--\ref{ass:H2} are imposed on each
$Q^i$ separately. Every estimate above is uniform in $x$ and applies verbatim
to each component, so Theorem~\ref{thm:main} extends to the interacting system
without modification.

A more striking consequence concerns the timing of the transitions. In the
original model the walks act synchronously: at each discrete time every walk
performs a move. Suppose instead that walk $i$ moves only at the epochs of an
independent geometric clock of parameter $p^i \in (0,1]$, remaining at its
current vertex otherwise. Averaging over the clock, the effective one-step
transition matrix becomes
\[
Q^i(x, p^i) = p^i\,Q^i(x) + (1-p^i)\,I,
\]
which interpolates between the active dynamics $Q^i(x)$ at $p^i = 1$ and pure
laziness at $p^i \to 0$. This family preserves the standing hypotheses:
$Q^i(x,p^i)$ has the same off-diagonal support as $Q^i(x)$, hence the same
communication classes and irreducibility; its diagonal is bounded below by
$1 - p^i$, which makes it aperiodic; and it is Lipschitz in $x$ with constant
at most $p^i L_Q \leq L_Q$. Most importantly, the invariant measure is left
untouched, since $\mu Q^i(x) = \mu$ gives
\[
\mu Q^i(x, p^i) = p^i\,\mu Q^i(x) + (1-p^i)\,\mu = \mu.
\]

Because the invariant measures, and therefore the mean field
$F(x) = -x + \pi(x)$ and its equilibria, do not depend on the parameters
$p^i$, the set of almost-sure accumulation points of the occupation process
$X(n)$ is the same for every choice of strictly positive geometric rates as it
is in the synchronous model. Introducing asynchrony or laziness into the walks
alters only the speed at which the process settles, not the configurations to
which it can converge. This invariance parallels the phenomenon observed for
interacting reinforced walks with independent geometric transition
times~\cite{PR2025, RPP22}, and it is perhaps the most economical illustration
of the robustness of the limiting behaviour under changes in the transition
schedule.

What the geometric clocks do change is the law of the process before it
stabilises, and in particular the manner in which mass is distributed among the
accumulation points. Characterising the resulting limiting distribution of the
occupation proportions---how the probabilities of the several stable equilibria
depend on the rates $p^1, \ldots, p^m$---lies beyond the reach of the
present techniques and remains an interesting open problem, likely more
tractable in the context of specific models through the analysis of
particular choices of~$Q^i(x)$.

%==============================================================================
% References
%==============================================================================

\end{document}